\documentclass{article}
\usepackage{amsmath,amssymb,amsthm}
\usepackage{parskip}
\usepackage{authblk}
\usepackage{tikz}
\usepackage[numbers]{natbib}
\usepackage[margin=3cm]{geometry}
\usepackage{multicol}
\usepackage[bookmarks=false,colorlinks=true,citecolor=blue]{hyperref}

\newtheorem{theorem}{Theorem}{\bfseries}{\itshape}
\newtheorem{definition}{Definition}{\bfseries}{\itshape}
\newtheorem{lemma}{Lemma}{\bfseries}{\itshape}
{\bfseries}{\itshape}
{\bfseries}{\itshape}
\newtheorem{claim}{Claim}{\bfseries}{\itshape}
\newtheorem{subclaim}{Subclaim}
\numberwithin{subclaim}{claim}
\newtheorem{case}{Case}
\newtheorem{subcase}{Case}
\numberwithin{subcase}{case}
\newtheorem{subsubcase}{Case}
\numberwithin{subsubcase}{subcase}

\title{\textbf{Upper Bounds on the Acyclic Chromatic Index of Degenerate Graphs}}
\author{\textbf{Nevil~Anto}}
\author{\textbf{Manu~Basavaraju}}
\author{\textbf{Suresh~Manjanath~Hegde}}
\author{\textbf{Shashanka~Kulamarva}}
\affil{National Institute of Technology Karnataka, Surathkal-575025, India\\ Email: \texttt{nevil.197cs005@nitk.edu.in, manub@nitk.edu.in, smhegde@nitk.edu.in, skulamarva.187ma007@nitk.edu.in}}
\date{}

\begin{document}
	\maketitle
	\begin{abstract}
		\noindent An acyclic edge coloring of a graph is a proper edge coloring without any bichromatic cycles. The \emph{acyclic chromatic index} of a graph $G$ denoted by $a'(G)$, is the minimum $k$ such that $G$ has an acyclic edge coloring with $k$ colors. Fiam\v{c}\'{\i}k \cite{Fiamvcik1978AC} conjectured that $a'(G) \le \Delta+2$ for any graph $G$ with maximum degree $\Delta$. A graph $G$ is said to be \emph{$k$-degenerate} if every subgraph of $G$ has a vertex of degree at most $k$. Basavaraju and Chandran \cite{Basavaraju2010AEC2deg} proved that the conjecture is true for $2$-degenerate graphs. We prove that for a $3$-degenerate graph $G$, $a'(G) \le \Delta+5$, thereby bringing the upper bound closer to the conjectured bound. We also consider $k$-degenerate graphs with $k \ge 4$ and give an upper bound for the acyclic chromatic index of the same.\\
		
		
	\end{abstract}
	
	\section{Introduction}
	Only finite and simple graphs are considered throughout this paper. Let $G=(V,E)$ be a graph with the vertex set $V$ and the edge set $E$. A \emph{path} in $G$ is a sequence of distinct vertices in $V$ such that there is an edge between every pair of consecutive vertices in the sequence. If we add an edge between the starting vertex and the ending vertex of a path in $G$, then the resulting structure is called a \emph{cycle} in $G$. Let $C$ be the given set of colors. A \emph{proper edge coloring} of $G$, with $C$, is a function $f:E \rightarrow C$ such that $f(e_1) \ne f(e_2)$ whenever $e_1$ and $e_2$ are adjacent to each other. The minimum number of colors required for a proper edge coloring of a given graph $G$ is called the \emph{chromatic index} of $G$ which is denoted by $\chi'(G)$. A proper edge coloring of $G$ is said to be an \emph{acyclic} edge coloring if there are no bichromatic cycles (cycles colored with exactly 2 colors) in $G$. The \emph{acyclic chromatic index} (also called \emph{acyclic edge chromatic number}) of a graph $G$ is the minimum number of colors required for an acyclic edge coloring of $G$ and is denoted by $a'(G)$. \citet{Grunbaum1973ACP} introduced the concept of acyclic coloring. The vertex analog of the acyclic chromatic index can be used to bound other parameters like oriented chromatic number \cite{Kostochka1997AC} and star chromatic number \cite{Fertin2004StarCol} of a graph. Both of these parameters have many practical applications including wavelength routing in optical networks \cite{Amar2001Ntwk}. By Vizing’s theorem \cite{Diestel2017GT}, we have $\Delta \le \chi'(G) \le \Delta+1$ where $\Delta=\Delta(G)$ is the maximum degree of a vertex in the graph $G$. Since acyclic edge coloring is also a proper edge coloring by definition, we have $a'(G) \ge \chi'(G) \ge \Delta$.
	
	It was conjectured by \citet{Fiamvcik1978AC} (and independently by \citet{Alon2001ACI}) that for any graph $G$, $a'(G) \le \Delta+2$. For an arbitrary graph $G$, the best-known upper bound for $a'(G)$ till date is $3.569(\Delta-1)$ given by \citet{Fialho2020AECBound}. They obtained it by using probabilistic techniques. This bound being far from the conjectured bound reflects the difficulty level of the problem.
	
	However, the conjecture has been proved for some special classes of graphs. \citet{Alon2001ACI} proved that there exists a constant $k$ such that $a'(G) \le \Delta+2$ for any graph $G$ with girth at least $k\Delta\log\Delta$. The acyclic chromatic index was exactly determined for some classes of graphs like series-parallel graphs when $\Delta \ne 4$ (\citet{Wang2011ACIK4MinorFree}), outerplanar graphs when $\Delta \ne 4$ (\citet{Hou2013AECOuterPlanarErr}, \citet{Hou2010AECOuterPlanar}), cubic graphs (\citet{Andersen2012AECCubic}), planar graphs with $\Delta \ge 4.2 \times 10 ^{14}$ (\citet{Cranston2019AECPlanar}) and planar graphs with girth at least 5 and $\Delta \ge 19$ (\citet{Basavaraju2011AECPlanar}). In the case of outerplanar graphs and series-parallel graphs, if $\Delta \ge 5$, then $a'(G)=\Delta$ and when $\Delta=3$, they characterize the graphs that require $4$ colors for the acyclic edge coloring.
	
	A graph $G$ is said to be \emph{$k$-degenerate} if every subgraph of $G$ has a vertex of degree at most $k$. It is easy to see that the acyclic edge coloring conjecture is true for $1$-degenerate graphs since a $1$-degenerate graph can be edge colored using exactly $\Delta$ colors. \citet{Basavaraju2010AEC2deg} proved that the conjecture is true for $2$-degenerate graphs by giving a strong upper bound of $\Delta+1$. Particularly, they prove that $a'(G) \le \Delta+1$, for a $2$-degenerate graph $G$.
	
	\citet{Fiedorowicz2011AECEdgeBound} proved that $a'(G) \le (t-1)\Delta + p$ for every graph $G$ which satisfies the condition that $|E(H)| \le t|V(H)|-1$ for every subgraph $H \subseteq G$, where $t \ge 2$ is a given integer, and the constant $p = 2t^3-3t+2$. One can verify that the class of $k$-degenerate graphs is a subclass of the class of graphs defined by \citet{Fiedorowicz2011AECEdgeBound}. Therefore, we can obtain an upper bound on the acyclic chromatic index of a $k$-degenerate graph $G$ as $a'(G) \le (k-1)\Delta + 2k^3-3k+2$ as in \cite{Fiedorowicz2011AECEdgeBound}. In this paper, we study $k$-degenerate graphs and improve this upper bound for the acyclic chromatic index of $k$-degenerate graphs. This upper bound is stated in the form of the following theorem:
	
	\begin{theorem}\label{thm:ACIkdeg}
		Let $G$ be a $k$-degenerate graph with $k \ge 4$ and maximum degree $\Delta$. Then $a'(G) \le \lceil(\frac{k+1}{2})\Delta\rceil + 1$.
	\end{theorem}
	
	We also come up with an upper bound for the acyclic chromatic index of $3$-degenerate graphs. Even though this does not prove the conjecture, it brings the upper bound close to the conjectured value. In particular, we prove the following theorem:
	
	\begin{theorem}\label{thm:ACI3deg}
		Let $G$ be a $3$-degenerate graph with maximum degree $\Delta$. Then $a'(G) \le \Delta+5$.
	\end{theorem}
	
	\section{Preliminaries}
	Let $G=(V,E)$ be a graph with $n$ vertices and $m$ edges. The \emph{degree} of a vertex in $G$ is the number of edges that are incident to that vertex in $G$. The degree of a vertex $v$ is represented as $deg_G(v)$. The minimum degree and the maximum degree of $G$ are represented as $\delta(G)$ and $\Delta(G)$ respectively. For any vertex $v \in V$, $N_G(v)$ is the set of all vertices in $V$ that are adjacent to the vertex $v$ in $G$. So, $N_G(v)$ represents the set of all neighbors of the vertex $v$ in $G$. Throughout the paper, we ignore $G$ in the above notations whenever the graph $G$ is understood from the context.
	
	Let $X \subseteq E$ and $Y \subseteq V$. The subgraph of $G$ obtained from the vertex set $V$ and the edge set $E \setminus X$ is denoted as $G \setminus X$. Similarly, $G \setminus Y$ is the subgraph of $G$ obtained from the vertex set $V \setminus Y$ and the edge set $E \setminus \{e \in E \mid \exists y \in Y \text{ such that } e \text{ is incident to } y\}$. If either $X$ or $Y$ is a singleton set $\{u\}$, then we just use $G \setminus u$ instead of $G \setminus \{u\}$. The subgraph of $G$ induced by the edges in $X$ is denoted by $G[X]$, i.e., $G[X]=(V_{X},E_{X})$ is a graph where $V_{X}=\{v \in V \mid \exists e \in X \text{ with $e$ incident on $v$}\}$ and $E_{X} = X$. Further notations and definitions can be found in \cite{West2001IGT}. We use the word coloring instead of acyclic edge coloring at some obvious places when there is no ambiguity.
	
	Further, we will mention some definitions and lemmas that are useful for our proof. These were given by \citet{Basavaraju2010AEC2deg}.
	
	\begin{definition}[\cite{Basavaraju2010AEC2deg}]
		Let $H$ be a subgraph of a graph $G$. An edge coloring $f$ of $H$ is called a \textbf{partial edge coloring} of $G$.
	\end{definition}
	
	An edge coloring of $G$ is also a partial edge coloring of $G$ since $G$ is also a subgraph of itself. A partial edge coloring $f$ of $G$ corresponding to a subgraph $H$ is said to be proper (and acyclic) if it is proper (and acyclic) in the subgraph $H$. Note that with respect to a partial coloring $f$, for an edge $e$, $f(e)$ may or may not be defined. So, whenever we use $f(e)$ for some edge $e$, we implicitly assume that $f(e)$ is defined. Let $f$ be a partial edge coloring of the graph $G$. For any vertex $x \in V$, we define $F_x(f) = \{f(xy) \mid y \in N_G(x)\}$. For any edge $uv \in E$, we define $F_{uv}(f) = F_v(f) \setminus \{f(uv)\}$. Whenever the partial coloring $f$ is understood from the context, we use $F_x$ and $F_{uv}$ instead of $F_x(f)$ and $F_{uv}(f)$. One can see that $F_{uv}$ is different from $F_{vu}$.
	
	\begin{definition}[\cite{Basavaraju2010AEC2deg}]
		An $(\alpha,\beta)$-maximal bichromatic path with respect to a partial coloring $f$ of $G$ is a maximal path in $G$ consisting of edges that are colored using the colors $\alpha$ and $\beta$ alternatingly. An $(\alpha,\beta,u,v)$-\textbf{maximal bichromatic path} is an $(\alpha,\beta)$-maximal bichromatic path which starts at the vertex $u$ with an edge colored with $\alpha$ and ends at the vertex $v$.
	\end{definition}
	
	Now, we mention a lemma that follows from the definition of acyclic edge coloring. We assume this lemma implicitly further down the paper. This lemma was mentioned as a fact in \cite{Basavaraju2010AEC2deg}.
	
	\begin{lemma}[\cite{Basavaraju2010AEC2deg}]\label{lem:UniqueMaxBichPath}
		Given a pair of colors $\alpha$ and $\beta$ in a proper coloring $f$ of $G$, there is at most one $(\alpha,\beta)$-maximal bichromatic path containing a particular vertex $v$ in $G$, with respect to $f$.
	\end{lemma}
	
	\begin{definition}[\cite{Basavaraju2010AEC2deg}]
		If the vertices $u$ and $v$ are adjacent in the graph $G$, then an $(\alpha,\beta,u,v)$-maximal bichromatic path in $G$, which ends at the vertex $v$ with an edge colored $\alpha$, is said to be an $(\alpha,\beta,uv)$-\textbf{critical path} in $G$.
	\end{definition}
	
	\begin{definition}[\cite{Basavaraju2010AEC2deg}]
		Let $f$ be a partial coloring of $G$. Let $u,a,b \in V$ and $ua,ub \in E$. A \textbf{color exchange} with respect to the edges $ua$ and $ub$ is defined as the process of obtaining a new partial coloring $g$ from the current partial coloring $f$ by exchanging the colors of the edges $ua$ and $ub$. The color exchange defines $g$ as follows. $g(ua)=f(ub)$, $g(ub)=f(ua)$ and for all other edges $e$ in $G$, $g(e)=f(e)$. The color exchange with respect to the edges $ua$ and $ub$ is said to be proper if the coloring obtained after the exchange is proper. The color exchange is said to be \textbf{valid} if the coloring obtained after the exchange is acyclic.
	\end{definition}
	
	A color $\gamma$ is said to be a \emph{candidate} color for an edge $e$ in $G$ with respect to a partial coloring $f$ if none of the edges that are incident on $e$ are colored $\gamma$. A candidate color $\gamma$ is said to be \emph{valid} for an edge $e$ if assigning the color $\gamma$ to $e$ does not result in any new bichromatic cycle in $G$. \citet{Basavaraju2010AEC2deg} mentioned the following lemma as a fact since it is obvious.
	
	\begin{lemma}[\cite{Basavaraju2010AEC2deg}]\label{lem:ColorValidity}
		Let $f$ be a partial coloring of $G$. A candidate color $\gamma$ is not valid for an edge $e=(u,v)$ if and only if there exists a color $\eta \in F_{uv} \cap F_{vu}$ such that there is a $(\eta,\gamma,uv)$-critical path in $G$ with respect to the coloring $f$.
	\end{lemma}
	
	Now, we state and prove a lemma on the availability of a special edge in a $k$-degenerate graph. This special edge that we obtain by the lemma is useful in our proof technique.
	
	\begin{lemma}\label{lem:SpecialEdgekdeg}
		If $G$ is a k-degenerate graph, then there exists an edge $xy$ in $G$ such that $deg(x) \le k$ and at most $k$ neighbors of $y$ have their degree strictly greater than $k$.
	\end{lemma}
	
	\begin{proof}
		Let $G$ be the given $k$ degenerate graph. By definition of $G$, there exists an edge $xy$ in $G$ such that $deg(x) \le k$. By way of contradiction assume that for every edge $xy$ in $G$ with $deg(x) \le k$, at least $k+1$ neighbors of $y$ have their degree strictly greater than $k$. Now, obtain a graph $G'$ by deleting all the vertices of degree at most $k$ from $G$. Clearly, $G'$ has some edges in it because the edges between the vertex $y$ and any of its higher degree neighbors will still be present in $G'$.
		
		Since $G'$ is a subgraph of $G$, we know that $G'$ is also a $k$ degenerate graph. Hence, there exists an edge $uv$ in $G'$ such that $deg(u) \le k$. If the degree of $u$ was at most $k$ in the graph $G$, then by choice of $G'$, the vertex $u$ should have been deleted while obtaining $G'$ from $G$. Since $u$ is present in $G'$, we are sure that the degree of $u$ was at least $k+1$ in $G$. Hence, there should exist a vertex $w$ which is a neighbor of $u$ in $G$ but $w \notin V(G')$.
		
		Since $w \notin V(G')$, $w$ was deleted while obtaining $G'$ from $G$, implying that $deg_{G}(w) \le k$. In fact, any neighbor of $u$ in $G$ but not present in $G'$ is of degree at most $k$ in $G$. Therefore, the number of neighbors of $u$ that have their degree at least $k+1$ is at most $deg_{G'}(u)$. Since $deg_{G'}(u) \le k$, we have an edge $wu$ in $G$ with $deg_{G}(w) \le k$ and at most $k$ neighbors of $u$ have their degree strictly greater than $k$ in $G$, a contradiction to our initial assumption. Thus we can conclude that the Lemma is valid.
	\end{proof}
	
	\section{Proof of Theorem~\ref{thm:ACIkdeg}}\label{prf:ACIkdegThm}
	\begin{proof}
		Let $G$ be a minimum counterexample to Theorem~\ref{thm:ACIkdeg} with respect to the number of edges. Let $G$ be a $k$-degenerate graph with $n$ vertices, $m$ edges and maximum degree $\Delta$. We also have $k \ge 4$. Let us define the number $p$ as follows:
		\begin{equation*}
			p = \lceil(\frac{k+1}{2})\Delta\rceil+1
		\end{equation*}
		Notice that $p$ is exactly the upper bound in Theorem~\ref{thm:ACIkdeg} that we intend to prove. Let $xy$ be an edge in $G$ such that $deg(x) \le k$. Such an edge $xy$ exists because $G$ is a $k$-degenerate graph. Let $G' = G \setminus xy$, i.e., a graph formed by deleting the edge $xy$ from $G$. Observe that $G'$ is also a $k$-degenerate graph and has less than $m$ edges. Since we did not add any edge or any vertex while obtaining $G'$ from $G$, we have $\Delta(G') \le \Delta(G)$. Therefore, since $G$ is a minimum counterexample, we have an acyclic edge coloring $g$ of $G'$ with $p$ colors. Let $C$ be the set of colors used in the coloring $g$, i.e., $C = \{1,2,\dots,p\}$.
		
		Now, we try to extend $g$ to an acyclic edge coloring $f$ of $G$ by assigning a color to the edge $xy$ from $C$, thereby arriving at a contradiction to the fact that $G$ is a minimum counterexample. Now, we define a set of vertices $S$ as follows:
		\begin{equation*}
			S = \{u \in N(x) \setminus y \mid \exists v \in N(y) \text{ such that } g(xu) = g(yv)\}
		\end{equation*}
		Let $E^*$ be the set of all edges in $G$ which are incident on at least one vertex in $S \cup \{x,y\}$. Observe that all the edges in $E^*$ except $xy$ are colored in $g$. Let $g(E^*)$ be the set of all colors seen on the edges in $E^*$ in the coloring $g$, excluding the repetitions. Now, we make the following claim about the validity of the colors which are not in $g(E^*)$.
		
		\begin{claim}\label{clm:ColorsNotInGTildeValid}
			Any color that is not in $g(E^*)$, is a valid color for the edge $xy$ in $G$.
		\end{claim}
		
		\begin{proof}
			Let $\alpha$ be a color that is not in $g(E^*)$. Then clearly $\alpha \notin F_{xy}$ and $\alpha \notin F_{yx}$ by choice of $E^*$ and $\alpha$. Hence, $\alpha$ is a candidate color for the edge $xy$ in $G$. By way of contradiction, assume that the candidate color $\alpha$ is not a valid color for the edge $xy$ in $G$. This means that there exists a color $\beta$ such that a $(\beta,\alpha,xy)$-critical path exists in $G'$. Since the $(\beta,\alpha,xy)$-critical path should be colored with the colors $\beta$ and $\alpha$ only, there should exist three vertices $x'$ and $x''$ and $y'$ in $G'$ distinct from $x$ and $y$ such that $g(xx') = \beta$, $g(x'x'') = \alpha$ and $g(yy') = \beta$. Observe that $x' \in N(x)$. But we also have $g(xx') = g(yy') = \beta$ implying that $x' \in S$.
			
			Therefore, $x'x'' \in E^*$. Since $g(x'x'') = \alpha$, this is a contradiction to our initial assumption that $\alpha$ was a color that is not in $g(E^*)$. Hence, we can conclude that our assumption was wrong and the claim holds, as desired.
		\end{proof}
		
		\begin{figure}[h]
			\centering
			\begin{tikzpicture}
				\begin{scope}[every node/.style={circle,draw,inner sep=0pt, minimum size=5ex}]
					\node (y) at (1,0) {$y$};
					\node (x) at (4,0) {$x$};
					\node (y1) at (-1,1) {$y_1$};
					\node (y2) at (-1,-1) {$y_2$};
					\node (x2) at (6,1) {$x_2$};
					\node (x3) at (4.25,2.25) {$x_1$};
				\end{scope}
				\node (x1) at (5,2) {};
				\node (x4) at (6.5,0) {};
				\node (x5) at (6.25,-0.5) {};
				\node (x6) at (5,-2) {};
				\node (x8) at (5.125,-0.25) {};
				\node (x9) at (5.5,0) {};
				\node (x10) at (4.7,-1.25) {};
				\node (x11) at (4.5,1) {};
				\node (x12) at (5,0.5) {};
				\node (x7) at (4.5,-1) {};
				\node (y3) at (0,0.5) {};
				\node (y4) at (0,-0.5) {};
				\draw (x) to (y);
				\draw (y) to (y1);
				\draw (y) to (y2);
				\draw (x) to (x3);
				\draw (x) to (x1);
				\draw (x) to (x2);
				\draw (x) to (x4);
				\draw (x) to (x6);
				\draw (x) to (x5);
				\draw [dashed] (x8) to (x7);
				\draw [dashed, bend left] (x3) to node[right]{$S$} (x2);
				\draw [dashed, bend left] (x4) to node[right]{$N_{G'}(x) \setminus S$} (x6);
				\draw [dashed] (x11) to (x12);
				\draw [dashed] (y3) to (y4);
			\end{tikzpicture}
			\caption{Neighborhood of the edge $xy$ in $G$}
			\label{fig:k-deg}
		\end{figure}
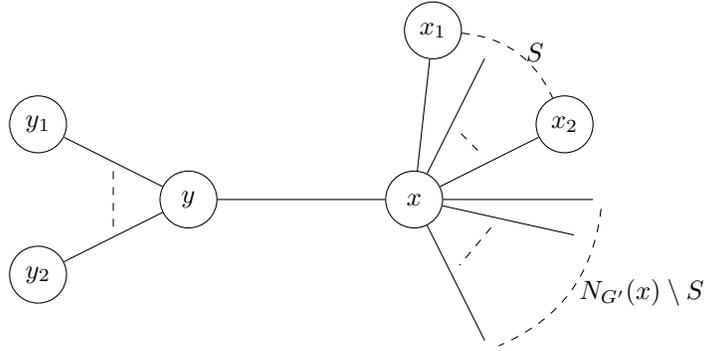
		
		Let $C'$ be the set of candidate colors for the edge $xy$ and let $C''$ be the set of colors in $F_{xy} \cup F_{yx}$. Observe that $C = C' \cup C''$. Note that any color in $C'$ is not valid for the edge $xy$ since $G$ is a minimum counterexample. Hence, together with Claim~\ref{clm:ColorsNotInGTildeValid}, we have that every color in $C$ is present in $g(E^*)$, i.e., $|g(E^*)| = |C' \cup C''| = p$. This also implies that every color in $C'$ is present at some vertex in $S$. Let the set of colors in $C'$ which appear only once on the edges which are incident on some vertex in $N_{G'}(x)$ be denoted by $C^*$. Notice that $|S| = |F_{xy} \cap F_{yx}|$. Now, we claim that the size of the set $S$ has a lower bound as follows:
		
		\begin{claim}\label{clm:SHasAtLeastKByTwo}
			$|S| = |F_{xy} \cap F_{yx}| > \frac{k-3}{2}$.
		\end{claim}
		
		\begin{proof}
			By way of contradiction, assume that $|S| = |F_{xy} \cap F_{yx}| = q \le \frac{k-3}{2}$. Since $deg(x) \le k$ and there are $q$ colors common in $F_{xy}$ and $F_{yx}$, we have:
			\begin{equation*}
				|C''| = |F_{xy} \cup F_{yx}| \le \Delta - 1 + k - q - 1 = \Delta+k-q-2
			\end{equation*}
			The remaining colors in $g(E^*)$ are seen at the vertices in $S$. Since $|S| = q$, we have $|C'| \le q(\Delta-1)$. Thus we have the following inequality:
			\begin{align*}
				|C' \cup C''| & \le (q(\Delta-1)) + (\Delta+k-q-2) \\
				& \le (q+1)\Delta + k - 2q - 2
			\end{align*}
			Since $k \le \Delta$, we have $k - 2q - 2 \le \Delta - 2$. Further, since $q \le \frac{k-3}{2}$, we have $q+1 \le \frac{k-1}{2}$. Therefore, the inequality becomes as follows:
			\begin{align*}
				|C' \cup C''| & \le (q+1)\Delta + k - 2q - 2 \\
				& \le (\frac{k-1}{2})\Delta + \Delta - 2 \\
				& \le (\frac{k+1}{2})\Delta - 2 \\
				& < p
			\end{align*}
			Observe that we have obtained an inequality $|g(E^*)| = |C' \cup C''| < p$, which is a contradiction to our assumption that $|g(E^*)| = p$. Therefore, the claim holds.
		\end{proof}
		
		Now, since $|S| > \frac{k-3}{2}$ by Claim~\ref{clm:SHasAtLeastKByTwo}, we have $|N_{G'}(x) \setminus S| \le k-1-(\frac{k-3}{2}) = \frac{k+1}{2}$. Hence, we have:
		\begin{equation*}
			|C''| = |F_{xy} \cup F_{yx}| \le \Delta - 1 + \frac{k+1}{2} = \Delta + \frac{k}{2} - \frac{1}{2}
		\end{equation*}
		Further, we claim that the cardinality of the set $C^*$ has a lower bound as follows:
		
		\begin{claim}\label{clm:TildeCAtLeast2}
			$|C^*| \ge 2$.
		\end{claim}
		
		\begin{proof}
			By way of contradiction, assume that $|C^*| = q \le 1$. Except for the $q$ colors in $C^*$, any of the remaining colors in $C'$ appear at least twice at the edges incident on the vertices in $N_{G'}(x)$. Therefore, the number of candidate colors that are not valid for the edge $xy$ (which is exactly given by $|C'|$) is at most
			\begin{equation*}
				\frac{(k-1)(\Delta-1)-q}{2} + q
			\end{equation*}
			Thus we have the following inequality:
			\begin{align*}
				|C' \cup C''| & \le (\frac{(k-1)(\Delta-1)-q}{2} + q) +  (\Delta + \frac{k}{2} - \frac{1}{2})\\
				& \le (\frac{k+1}{2})\Delta + \frac{q}{2} \\
				& \le (\frac{k+1}{2})\Delta + \frac{1}{2} \\
				& < p
			\end{align*}
			Since $|g(E^*)| = |C' \cup C''|$, we have $|g(E^*)| < p$, a contradiction to our initial assumption that $|g(E^*)| = p$. Therefore, our assumption that $|C^*| \le 1$ is not valid and the claim holds.
		\end{proof}
		
		Now, we make the following claim about the number of vertices in $S$ whose edges see the colors in $C^*$.
		
		\begin{claim}\label{clm:TwoVertInSSeeTildeCColors}
			There exist at least two vertices in $S$ whose edges see the colors in $C^*$.
		\end{claim}
		
		\begin{proof}
			Every color in $C^*$ is present on some edge incident to a vertex in $S$, because $C^* \subseteq C'$ and every color in $C'$ is present on some edge incident to a vertex in $S$. By way of contradiction, assume that every color in $C^*$ is present on the edges incident to a single vertex in $S$. Let $x'$ be the vertex in $S$ such that every color in $C^*$ is in $F_{xx'}$ and let $g(xx') = \zeta$. Let $\alpha$ and $\beta$ be two colors in $C^*$. By Claim~\ref{clm:TildeCAtLeast2}, $\alpha$ and $\beta$ exist. Since no color in $C^*$ is valid for the edge $xy$ in $G$, for each color $\gamma \in C^*$, there exists a $(\zeta,\gamma,xy)$-critical path in $G$.
			
			\begin{subclaim}\label{subclm:C'MinusFxxTildeIsNonEmpty}
				$C' \setminus F_{xx'} \neq \emptyset$.
			\end{subclaim}
			
			\begin{proof}
				By way of contradiction, assume that $C' \setminus F_{xx'} = \emptyset$. This means that every candidate color for the edge $xy$ is in $F_{xx'}$, implying that $|C'| \le |F_{xx'}| \le \Delta-1$. Further, we also have $|C''| \le \Delta + \frac{k}{2} - \frac{1}{2}$. Therefore, we have the following inequality:
				\begin{align*}
					|C' \cup C''| & \le (\Delta-1) +  (\Delta + \frac{k}{2} - \frac{1}{2})\\
					& \le 2\Delta + \frac{k}{2} - \frac{3}{2}
				\end{align*}
				Observe that since $4 \le k \le \Delta$, we have $p = \lceil(\frac{k+1}{2})\Delta\rceil+1 \ge \lceil2.5\Delta\rceil+1$. If $k=4$, then we have:
				\begin{equation*}
					|C' \cup C''| \le 2\Delta + \frac{4}{2} - \frac{3}{2} = 2\Delta+0.5 < p
				\end{equation*}
				Otherwise, if $k \ge 5$, then we have $p = \lceil(\frac{k+1}{2})\Delta\rceil+1 \ge \lceil3\Delta\rceil+1$, which in turn implies that:
				\begin{equation*}
					|C' \cup C''| \le 2\Delta + \frac{k}{2} - \frac{3}{2} < p
				\end{equation*}
				Therefore, in any case, for $k \ge 4$, we have $|C' \cup C''| < p$, a contradiction to the fact that $|C| = |C' \cup C''| = p$. Hence, our assumption that $C' \setminus F_{xx'} = \emptyset$ was wrong and the subclaim holds.
			\end{proof}
			
			Now, assume that there exists a color $\gamma$ in $C' \setminus F_{xx'}$ that repeats at most twice on the edges incident on $N_{G'}(x) \setminus x'$. Since every color in $C^*$ is in $F_{xx'}$, we have $\gamma \notin C^*$, implying that $\gamma$ repeats exactly twice on the edges incident on $N_{G'}(x) \setminus x'$. Let $x_1$ and $x_2$ be vertices in $N_{G'}(x) \setminus x'$ such that $\gamma \in F_{xx_1}$ and $\gamma \in F_{xx_2}$. Now, recolor the edges $xx_1$ and $xx_2$ with $\alpha$ and $\beta$ respectively. Observe that for any color $\eta$ in $C^*$, $\eta \notin F_{xv}$ for every $v \in N_{G'}(x) \setminus x'$. Hence, this is particularly true for the colors $\alpha$ and $\beta$ in $C^*$. Therefore, the recoloring is proper. Since for every color $\eta \in C^*$, the $(\zeta,\eta)$-bichromatic path in $G$ starting from $x$ ends at $y$, by Lemma~\ref{lem:UniqueMaxBichPath}, there is no new bichromatic cycle formed by this recoloring. Hence, the recoloring is valid. Now, since $\gamma \in (C' \setminus F_{xx'})$, $\gamma \notin \{\alpha,\beta\}$, which implies that $\gamma$ is a candidate color for the edge $xy$. Further, for any vertex $v \in N_{G'}(x) \setminus \{x_1,x_2\}$, we have $\gamma \notin F_{xv}$. Therefore, $\gamma$ is also valid for the edge $xy$, a contradiction since $G$ is a minimum counterexample.
			
			Hence, we can safely assume that there does not exist a color in $C' \setminus F_{xx'}$ that repeats at most twice on the edges incident on $N_{G'}(x) \setminus x'$. By Subclaim~\ref{subclm:C'MinusFxxTildeIsNonEmpty}, we have $C' \setminus F_{xx'} \neq \emptyset$. This means that every color in $C' \setminus F_{xx'}$ repeats at least three times on the edges incident on $N_{G'}(x) \setminus x'$. Therefore, we can infer the following:
			\begin{align*}
				|C' \setminus F_{xx'}| & \le \frac{(k-2)(\Delta-1)}{3} \\
				& \le (\frac{k-2}{3})\Delta - \frac{k}{3} + \frac{2}{3}
			\end{align*}
			Recall that we already have $|C''| \le \Delta + \frac{k}{2} - \frac{1}{2}$. Therefore, collectively we have the following inequality:
			\begin{align*}
				|C' \cup C''| & \le |C' \setminus F_{xx'}| + |F_{xx'}| + |C''| \\
				& \le ((\frac{k-2}{3})\Delta - \frac{k}{3} + \frac{2}{3}) + (\Delta-1) + (\Delta + \frac{k}{2} - \frac{1}{2}) \\
				& \le (\frac{k+4}{3})\Delta + \frac{k}{6} - \frac{5}{6}
			\end{align*}
			Notice that for some color in $C' \setminus F_{xx'}$ to repeat at least three times on the edges incident on $N_{G'}(x) \setminus x'$, it is necessary that $|N_{G'}(x) \setminus x'| \ge 3$. This implies that $k \ge 5$. Recall that we have $p \ge \lceil3\Delta\rceil+1$ whenever $k \ge 5$. If $k=5$, then we have:
			\begin{equation*}
				|C' \cup C''| \le (\frac{k+4}{3})\Delta + \frac{k}{6} - \frac{5}{6} = (\frac{5+4}{3})\Delta + \frac{5}{6} - \frac{5}{6} = 3\Delta < p
			\end{equation*}
			But this is a contradiction to the fact that $|C' \cup C''| = p$.
			
			Otherwise, let $k \ge 6$. Now, since $k \le \Delta$, we have $\frac{k}{6} \le \frac{\Delta}{6}$. Hence, we have the following:
			\begin{align*}
				|C' \cup C''| & \le (\frac{k+4}{3})\Delta + \frac{\Delta}{6} - \frac{5}{6} \\
				& \le (\frac{2k+9}{6})\Delta - \frac{5}{6}
			\end{align*}
			Notice that if $k \ge 6$, then $\frac{2k+9}{6} \le \frac{k+1}{2}$. Since we already have $k \ge 6$, we can infer that $|C' \cup C''| \le  (\frac{k+1}{2})\Delta - \frac{5}{6} < p$, a contradiction to the fact that $|C' \cup C''| = p$.
			
			Therefore, in any case, we arrive at a contradiction. Hence, our assumption that every color in $C^*$ is present on the edges incident to a single vertex in $S$, was wrong and the claim holds.
		\end{proof}
		
		Recall that by Claim~\ref{clm:TildeCAtLeast2}, we have that $|C^*| \ge 2$. Further, by Claim~\ref{clm:TwoVertInSSeeTildeCColors}, we are sure that there exist at least two vertices in $S$ whose edges see the colors in $C^*$. Let $x_1$ and $x_2$ be the vertices in $S$ such that there exist two colors $\gamma$ and $\eta$ in $C^*$ satisfying $\gamma \in F_{xx_1}$ and $\eta \in F_{xx_2}$. Let $g(xx_1) = \alpha$ and $g(xx_2) = \beta$.
		
		Since $\gamma$ and $\eta$ are not valid for the edge $xy$ in $G$, there exists an $(\alpha,\gamma,xy)$-critical path in $G$ and also a $(\beta,\eta,xy)$-critical path in $G$. Now, we recolor the edge $xx_2$ to $\gamma$. This is still a proper coloring since $\gamma \in F_{xx_1}$ implies that $\gamma \notin F_{xx_2}$ by choice of $\gamma$. Now, since the $(\alpha,\gamma)$-bichromatic path starting from $x$ ends at $y$, by Lemma~\ref{lem:UniqueMaxBichPath}, there is no new bichromatic cycle created indicating that the recoloring is valid. Observe that $\eta$ becomes a valid color for the edge $xy$ because by this recoloring we have eliminated the unique $xy$-critical path in $G$ that involves the color $\eta$, i.e., the $(\beta,\eta,xy)$-critical path in $G$ has been eliminated by this recoloring. Thus we can color the edge $xy$ with color $\eta$ and extend the coloring $g$ to a coloring $f$ of $G$ with $p$ colors. But this is a contradiction to the fact that $G$ is a minimum counterexample. Therefore, we can conclude that a minimum counterexample to Theorem~\ref{thm:ACIkdeg} does not exist which in turn implies the validity of Theorem~\ref{thm:ACIkdeg}.
	\end{proof}

	\section{Proof of Theorem~\ref{thm:ACI3deg}}\label{prf:ACI3degThm}
	\begin{proof}
		Let $G$ be the given $3$-degenerate graph with $n$ vertices, $m$ edges and maximum degree $\Delta$. We use induction on the number of edges $m$ of $G$ to proceed with the proof. Let $xy$ be an edge in $G$ such that $deg(x) \le 3$ and at most 3 neighbors of $y$ have their degree strictly greater than 3. The existence of such an edge $xy$ is guaranteed by Lemma~\ref{lem:SpecialEdgekdeg}. Further, we choose $x$ as the neighbor of $y$ that has the minimum degree among the vertices in $N(y)$. Let $G' =  G \setminus xy$, i.e., a graph formed by deleting the edge $xy$ from $G$. Observe that $G'$ is also $3$-degenerate and has less than $m$ edges. Further, we have $\Delta(G') \le \Delta(G)$. Hence, by induction, we have an acyclic edge coloring $g$ of $G'$ with $\Delta+5$ colors. Let $N'(y)$ be the set of all neighbors of $y$ in $G'$ having their degree less than or equal to 3 and let $N''(y)$ be the set of all neighbors of $y$ in $G'$ having their degree strictly greater than 3. Notice that we have $|N''(y)| \le 3$. Let $S$ be the set of colors in $F_y$ excluding those which belong to the set $\{g(yz) \mid z \in N''(y)\}$. Since $|N''(y)| \le 3$, we have $|F_y \setminus S| \le 3$.
		
		Now, we try to extend $g$ to an acyclic edge coloring $f$ of $G$ by assigning a color to the edge $xy$ from the available $\Delta+5$ colors. If $deg_G(x) = 1$, then by assigning the edge $xy$ any color other than the colors in $F_{xy}$, we can extend $g$ to the required coloring $f$ of $G$, since $|F_{xy}| \le \Delta - 1$. Thus we can assume that $deg(x) \ge 2$. Further, depending on the degree of the vertex $x$ in $G$, we have the following cases:
		
		\begin{case}\label{case:DegX2}
			$deg(x)=2$.
		\end{case}
		
		Let $x'$ be the unique neighbor of $x$ in $G'$. Let $g(xx') = \alpha$. If $\alpha \notin F_y$, then we can assign any color satisfying the proper coloring to the edge $xy$ and extend $g$ to the required coloring $f$ of $G$. Thus we can assume that $\alpha \in F_y$. Let $y'$ be the neighbor of $y$ such that $g(yy') = \alpha$. Observe that the candidate colors which are not valid for the edge $xy$ are precisely the colors in $F_{yy'}$. Further, since $\alpha \in F_{xy}$, the colors which are not candidate colors for the edge $xy$ are the colors in $F_{xy}$. Therefore, any color that is not in $F_{xy} \cup F_{yy'}$ is valid for the edge $xy$. Depending on whether $\alpha \in S$ or not, we have the following cases:
		
		\begin{subcase}\label{case:AlphaInSDegX2}
			$\alpha \in S$.
		\end{subcase}
		
		Since $\alpha \in S$, we have $deg(y') \le 3$ which implies that $|F_{yy'}| \le 2$. Therefore, we have $|F_{xy} \cup F_{yy'}| \le \Delta + 1$. We still have 4 colors available for the edge $xy$ and by using any one of those 4 colors, we can extend $g$ to the required coloring $f$ of $G$.
		
		\begin{figure}[h]
			\centering
			\begin{tikzpicture}
				\begin{scope}[every node/.style={circle,draw,inner sep=0pt, minimum size=5ex}]
					\node (x) at (1,0) {$x$};
					\node (y) at (4,0) {$y$};
					\node (x') at (-2,0) {$x'$};
				\end{scope}
				\node (y1) at (5,2) {};
				\node (y3) at (5.5,1.5) {};
				\node (y2) at (6,1) {};
				\node (y4) at (6.5,0) {};
				\node (y5) at (6.25,-0.5) {};
				\node (y6) at (5,-2) {};
				\node (y8) at (5.125,-0.25) {};
				\node (y9) at (5.5,0) {};
				\node (y10) at (4.7,-1.25) {};
				\node (y7) at (4.5,-1) {};
				\draw (x) to (y);
				\draw (x) to node[below]{$\alpha (\rightarrow \beta)$} (x');
				\draw (y) to node[above]{$\alpha$} (y1);
				\draw (y) to (y2);
				\draw (y) to (y3);
				\draw (y) to (y4);
				\draw (y) to (y6);
				\draw (y) to (y5);
				\draw [dotted] (y8) to (y7);
				\draw [dashed, , bend left] (y1) to node[right]{$N''(y)$} (y2);
				\draw [dashed, , bend left] (y4) to node[right]{$N'(y)$} (y6);
				\draw [dashed, , bend left] (y9) to node[right]{$S$} (y10);
			\end{tikzpicture}
			\caption{Neighborhood of the edge $xy$ in $G$ in Case~\ref{case:AlphaNotInSDegX2}}
			\label{fig:DegX2}
		\end{figure}
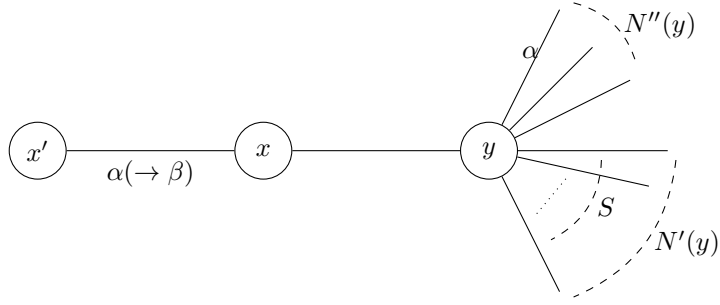
		
		\begin{subcase}\label{case:AlphaNotInSDegX2}
			$\alpha \notin S$.
		\end{subcase}
		
		Recall that we have $|F_{xx'}| \le \Delta - 1$ and $|F_y \setminus S| \le 3$. Therefore, we can infer the following:
		\begin{equation*}
			|F_{xx'} \cup (F_y \setminus S)| \le \Delta + 2
		\end{equation*}
		Now, we pick a color $\beta$ that is not present in the set $F_{xx'} \cup (F_y \setminus S)$ and recolor the edge $xx'$ from $\alpha$ to $\beta$. The recoloring is proper since $\beta \notin F_{xx'}$. The recoloring is valid because the edge $xy$ is not yet colored. Observe that if $\beta \notin F_y$, then we can assign any color satisfying the proper coloring to the edge $xy$ and extend $g$ to the required coloring $f$ of $G$. Thus we can assume that $\beta \in F_y$. Further, since $\beta$ was picked satisfying $\beta \notin F_{xx'} \cup (F_y \setminus S)$, we can conclude that $\beta \in S$. This boils down to Case~\ref{case:AlphaInSDegX2} and hence, we are done.
		
		\begin{case}\label{case:DegX3}
			$deg(x)=3$.
		\end{case}
		
		Notice that for this case, any neighbor of $y$ which is not in $N''(y)$ has the degree exactly $3$ by the choice of the vertex $x$. Let $x_1$ and $x_2$ be the neighbors of $x$ in $G$ other than $y$. Let $g(xx_1) = \alpha$ and let $g(xx_2) = \beta$. We define $R$ to be the set of all colors from the total available $\Delta + 5$ colors that are not in $F_{xy} \cup F_{yx}$. If any color in $R$ is valid for the edge $xy$, then we can extend $g$ to a coloring $f$ of $G$, as desired. Thus we can assume that no color in $R$ is valid for the edge $xy$. Depending on whether the colors $g(xx_1)$ and $g(xx_2)$ belongs to $F_{xy} \setminus S$, we have the following cases:
		
		\begin{subcase}\label{case:NoColorInFxyMinusS}
			$\{g(xx_1),g(xx_2)\} \cap (F_{xy} \setminus S) = \emptyset$.
		\end{subcase}
		
		Recall that $g(xx_1) = \alpha$ and let $g(xx_2) = \beta$. For this case, no color in $\{\alpha,\beta\}$ is in $F_{xy} \setminus S$, implying that every color in $\{\alpha,\beta\}$ is either present in $S$ or not present in $F_{xy}$.
		
		If no color in $\{\alpha,\beta\}$ is in $S$, then it implies that no color in $\{\alpha,\beta\}$ is in $F_{xy}$. Hence, we can use any color satisfying the proper coloring for the edge $xy$ and extend $g$ to the required coloring $f$ of $G$.
		
		Otherwise, let exactly one color in $\{\alpha,\beta\}$ be in $S$. Without loss of generality, let $\alpha \in S$ which implies $\alpha \in F_{xy}$. Let $y_m$ be the neighbor of $y$ such that $g(yy_m) = \alpha$. Note that $y_m \in N'(y)$. Therefore, we have $|F_{yy_m}| \le 2$. Further, one can see that the set of candidate colors that are not valid for the edge $xy$ is given by $F_{yy_m}$ for this case. Since $\alpha \in F_{xy}$, we have $|F_{xy} \cup F_{yx}| \le \Delta$, which implies that $|F_{xy} \cup F_{yx} \cup F_{yy_m}| \le \Delta + 2$.
		
		Otherwise, let both the colors in $\{\alpha,\beta\}$ be in $S$. Hence, $\alpha \in S$ and $\beta \in S$ which implies that $\alpha \in F_{xy}$ and $\beta \in F_{xy}$. Let $y_m$ and $y_n$ be the neighbors of $y$ such that $g(yy_m) = \alpha$ and $g(yy_n) = \beta$. Note that $y_m \in N'(y)$ and $y_n \in N'(y)$. Therefore, we have $|F_{yy_m} \cup F_{yy_n}| \le 4$. Further, one can see that the set of candidate colors that are not valid for the edge $xy$ is given by $F_{yy_m} \cup F_{yy_n}$ for this case. Since both $\alpha$ and $\beta$ are in $F_{xy}$, we have $|F_{xy} \cup F_{yx}| \le \Delta-1$, which implies the following:
		\begin{align*}
			|F_{xy} \cup F_{yx} \cup F_{yy_m} \cup F_{yy_n}| & \le \Delta - 1 + 2 + 2 \\
			& = \Delta + 3
		\end{align*}
		Since we have a total of $\Delta + 5$ colors, we have a valid color $\gamma$ for the edge $xy$ in any case, irrespective of the number of common colors in $\{\alpha,\beta\}$ and $S$. By assigning $\gamma$ to $xy$, we can extend $g$ to the required coloring $f$ of $G$.
		
		\begin{subcase}\label{case:SomeColorInFxyMinusS}
			$\{g(xx_1),g(xx_2)\} \cap (F_{xy} \setminus S) \neq \emptyset$.
		\end{subcase}
		
		In this case, at least one color in $\{\alpha,\beta\}$ is in $F_{xy} \setminus S$. Now, we define a color in $S$ to be \emph{freeable} with respect to the edge $xy$ as follows.
		
		\begin{definition}
			For any vertex $y'$ in $N'(y)$, the color $g(yy')$ in $S$ is said to be freeable if we can recolor $g(yy')$ with a color in $R$ without forming any new bichromatic cycle.
		\end{definition}
		
		Observe that after this recoloring, $g(yy')$ becomes a candidate color for the edge $xy$ in $G$. Now, we make the following claim regarding the number of freeable colors in $S$.
		
		\begin{claim}\label{clm:2FreeableColors}
			There exists at most 2 colors in $S$ which are not freeable.
		\end{claim}
		
		\begin{proof}
			By way of contradiction, assume that there exist at least 3 colors in $S$ which are not freeable. Let those colors be $\gamma_1$, $\gamma_2$ and $\gamma_3$ such that $g(yy_1) = \gamma_1$, $g(yy_2) = \gamma_2$ and $g(yy_3) = \gamma_3$ for $y_1,y_2,y_3 \in N'(y)$. Throughout the proof of the claim, whenever we use $i$, we implicitly assume that for any $i$ with $1 \le i \le 3$. Since $\gamma_i \in S$, we have $y_i \notin N''(y)$. This implies that $deg(y_i)=3$. Let $y'_i$ and $y''_i$ be the neighbors of $y_i$ other than $y$ and let $g(y_iy'_i) = \nu_i$ and $g(y_iy''_i) = \eta_i$.
			
			If exactly one among $\alpha$ or $\beta$ is in $F_{xy}$, then we have $|F_{xy} \cup F_{yx}| \le \Delta$, which implies that $|R| \geq (\Delta + 5) - (\Delta) = 5$. Otherwise, if both $\alpha$ and $\beta$ are in $F_{xy}$, then we have $|F_{xy} \cup F_{yx}| \le \Delta - 1$, which implies that $|R| \geq (\Delta + 5) - (\Delta - 1) = 6$. Since we have assumed that at least one among $\alpha$ or $\beta$ belongs to $F_{xy}$, in any case, we have that $|R| \geq 5$. Let $\{\mu_1,\mu_2,\mu_3,\mu_4,\mu_5\}$ be any five colors in $R$.
			
			Since $g(yy_i) = \gamma_i$ is not freeable, it means that if we recolor $g(yy_i)$ with any color in $R$, a new bichromatic cycle will be formed. Therefore, we have that for every $\mu_j \in R$, either a $(\nu_i,\mu_j,yy_i)$-critical path exists or a $(\eta_i,\mu_j,yy_i)$-critical path exists or both the above critical paths exist in $G'$. Therefore, at least three out of five $yy_i$-critical paths involve $\nu_i$ or at least three out of five $yy_i$-critical paths involve $\eta_i$. Recall that the statement is true for any $i$ with $1 \le i \le 3$. Hence, without loss of generality, assume that at least three $yy_1$-critical paths involve $\nu_1$, at least three $yy_2$-critical paths involve $\nu_2$ and at least three $yy_3$-critical paths involve $\nu_3$. Observe that at least three $yy_i$-critical paths that involve $\nu_i$ should reach $y$ through a vertex $z_i \in N(y)$ with $deg(z_i) \ge 4$ which implies that $z_i \in N''(y)$ and $\nu_i \in F_{xy} \setminus S$.
			
			\begin{subclaim}\label{subclm:NuIsAreDistinct}
				The colors $\nu_1$, $\nu_2$ and $\nu_3$ are all distinct.
			\end{subclaim}
			
			\begin{proof}
				By way of contradiction, without loss of generality assume that $\nu_1 = \nu_2 = \nu$ for some color $\nu$. Let $y' \in N''(y)$ with $g(yy') = \nu$. Then there exist at least three $(\nu,\mu_j,yy_1)$-critical paths and at least three $(\nu,\mu_k,yy_2)$-critical paths for $\{\mu_1,\mu_2,\mu_3,\mu_4,\mu_5\}$ in $R$. Notice that we have five colors and at least six critical paths under consideration. Hence, we have a color $\mu_j$ in $R$ such that there exists a $(\nu,\mu_j,yy_1)$-critical path and a $(\nu,\mu_j,yy_2)$-critical path. This implies that there exists a $(\nu,\mu_j)$-maximal bichromatic path starting from the vertex $y$ ending at the vertex $y_1$ and there exists a $(\nu,\mu_j)$-maximal bichromatic path starting from the vertex $y$ ending at the vertex $y_2$, a contradiction to Lemma~\ref{lem:UniqueMaxBichPath}. Therefore, our assumption that $\nu_1 = \nu_2 = \nu$ is wrong and the subclaim holds.
			\end{proof}
			
			Since any color $\mu_j$ in $R$ with $1 \le j \le 5$, is not valid for the edge $xy$, there exists either an $(\alpha,\mu_j,xy)$-critical path or a $(\beta,\mu_j,xy)$-critical path or both in $G'$. Hence, there exist at least three $(\alpha,\mu_j,xy)$-critical paths or at least three $(\beta,\mu_j,xy)$-critical paths. Without loss of generality, assume the existence of at least three $(\alpha,\mu_j,xy)$-critical paths. Now, recall that we have $\nu_i \in F_{xy} \setminus S$. Since $|F_{xy} \setminus S| \le 3$, Subclaim~\ref{subclm:NuIsAreDistinct} implies that $\alpha$ is a color in $\{\nu_1, \nu_2, \nu_3\}$. Without loss of generality, let $\alpha = \nu_3$. Then there exists at least three $(\alpha,\mu_j,yy_3)$-critical paths together with the already assumed at least three $(\alpha,\mu_j,xy)$-critical paths for $\{\mu_1,\mu_2,\mu_3,\mu_4,\mu_5\}$ in $R$. Notice that we have five colors and at least six critical paths under consideration. Hence, we have a color $\mu_j$ in $R$ such that there exists an $(\alpha,\mu_j,yy_3)$-critical path and an $(\alpha,\mu_j,xy)$-critical path. This implies that there exists an $(\alpha,\mu_j)$-maximal bichromatic path starting from the vertex $y$ ending at the vertex $y_3$ and there exists an $(\alpha,\mu_j)$-maximal bichromatic path starting from the vertex $y$ ending at the vertex $x$, a contradiction to Lemma~\ref{lem:UniqueMaxBichPath}. Hence, our assumption that there exist at least 3 colors in $S$ which are not freeable is wrong and the claim holds.
		\end{proof}
		
		Let $S' \subset S$ be the set of all colors in $S$ which are not freeable. Now, we define the set $T$ to be $T = R \cup (S \setminus S')$. Further, we make the following claim regarding the cardinality of the set $T$.
		
		\begin{claim}\label{clm:AtLeastDeltaMinusOneCandidateColors}
			$|T| \geq \Delta - 1$.
		\end{claim}
		
		\begin{proof}
			Observe that the set $T$ is precisely the set of all colors that are not in $F_{yx} \cup (F_{xy} \setminus S) \cup S'$. By Claim~\ref{clm:2FreeableColors}, we have that $|S'| \le 2$. We also have $|F_{xy} \setminus S| \le 3$. Since $deg(x) = 3$, we have $|F_{yx}| = 2$. Precisely, $F_{yx} = \{\alpha,\beta\}$. But we have already assumed that at least one among $\alpha$ or $\beta$ belongs to $F_{xy} \setminus S$. Therefore, there exists at most one color in $F_{yx} = \{\alpha,\beta\}$ which is not in $F_{xy} \setminus S$. With all these observations we can infer the following:
			\begin{align*}
				|T| & = \Delta + 5 - |F_{yx} \cup (F_{xy} \setminus S) \cup S'| \\
				& \geq \Delta + 5 - (1 + 3 + 2) \\
				& = \Delta - 1
			\end{align*}
			Thus we have the lower bound for the set $T$, as claimed.
		\end{proof}
		
		Further, depending on how many colors among $\{g(xx_1),g(xx_2)\}$ belong to the set $F_{xy}$, we have the following cases:
		
		\begin{subsubcase}\label{case:OneColorInFxy}
			Exactly one color in $\{g(xx_1),g(xx_2)\}$ belongs to $F_{xy}$.
		\end{subsubcase}
		
		Recall that we have $g(xx_1)=\alpha$ and $g(xx_2)=\beta$. Without loss of generality, let $\alpha \in F_{xy}$ and $\beta \notin F_{xy}$. Let $y_1$ be the neighbor of $y$ in $G$ such that $g(yy_1) = \alpha$. Since we already have that at least one among $\alpha$ or $\beta$ belongs to $F_{xy} \setminus S$, $\beta$ not being present in $F_{xy}$ will imply that $\alpha \notin S$. Collectively, we can infer that $\alpha \in (F_{xy} \setminus S)$. Observe that since $\alpha \in F_{xy}$, we have that $|F_{xy} \cup F_{yx}| \le \Delta$ implying that $|R| \geq 5$. Let $\{\mu_1,\mu_2,\mu_3,\mu_4,\mu_5\} \in R$. If any color in $R$ is valid for the edge $xy$, then we are done. Hence, we can assume that every candidate color for the edge $xy$ in $R$ is not valid. This implies that for every color $\mu_i$ in $R$, there exists an $(\alpha,\mu_i,xy)$-critical path in $G'$ with respect to $g$.
		
		Now, by Claim~\ref{clm:AtLeastDeltaMinusOneCandidateColors}, we have $|T| \geq \Delta - 1$. If there exists a color $\zeta \in T$ such that there is no $(\alpha,\zeta,xy)$-critical path in $G'$ with respect to $g$, then we can free the color $\zeta$ if necessary and assign $\zeta$ to the edge $xy$ and thereby extend $g$ to the required coloring $f$ of $G$. Therefore, we can assume that for every color $\zeta \in T$, there exists an $(\alpha,\zeta,xy)$-critical path in $G'$ with respect to $g$.
		
		Let us assume that $\beta \in F_{xx_1}$. Note that $|F_{xx_1}| \le \Delta - 1$. Since $\beta \in F_{yx}$, we have $\beta \notin T$. This together with the assumption that $\beta \in F_{xx_1}$ implies that there exists a color $\eta$ such that $\eta \in T$ but $\eta \notin F_{xx_1}$. This implies that there can not be any $(\alpha,\eta,xy)$-critical path in $G'$ with respect to $g$. Since $\eta \in T$, this is a contradiction to our previous assumption that for every color $\zeta \in T$, there exists an $(\alpha,\zeta,xy)$-critical path in $G'$ with respect to $g$. Hence, our assumption that $\beta \in F_{xx_1}$ is not true, which implies that we are good to conclude that $\beta \notin F_{xx_1}$.
		
		Since $|F_{xx_1} \cup \{\beta\}| \le \Delta$ and $|F_{xy} \setminus S| \le 3$, we are sure that there exists a color $\gamma$ such that $\gamma \notin F_{xx_1} \cup \{\beta\} \cup (F_{xy} \setminus S)$. Now, we recolor the edge $xx_1$ with $\gamma$. This recoloring is valid since $\beta \notin F_{xx_1}$. If $\gamma \notin S$, then clearly $\gamma \notin F_{xy}$ which implies that by assigning any color to the edge $xy$ which satisfies proper coloring, we can extend $g$ to the required coloring $f$ of $G$. Otherwise, if $\gamma \in S$, then since $\beta \notin F_{xy}$, by Case~\ref{case:NoColorInFxyMinusS}, we are done.
		
		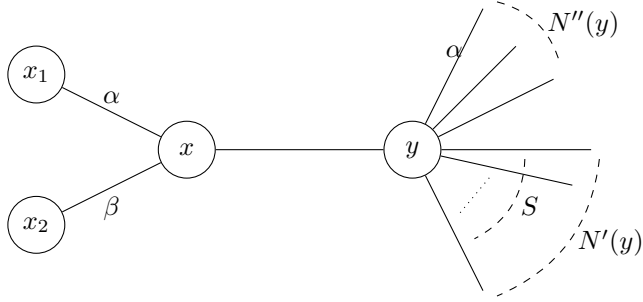
\begin{figure}[h]
			\centering
			\begin{tikzpicture}
				\begin{scope}[every node/.style={circle,draw,inner sep=0pt, minimum size=5ex}]
					\node (x) at (1,0) {$x$};
					\node (y) at (4,0) {$y$};
					\node (x1) at (-1,1) {$x_1$};
					\node (x2) at (-1,-1) {$x_2$};
				\end{scope}
				\node (y1) at (5,2) {};
				\node (y3) at (5.5,1.5) {};
				\node (y2) at (6,1) {};
				\node (y4) at (6.5,0) {};
				\node (y5) at (6.25,-0.5) {};
				\node (y6) at (5,-2) {};
				\node (y8) at (5.125,-0.25) {};
				\node (y9) at (5.5,0) {};
				\node (y10) at (4.7,-1.25) {};
				\node (y7) at (4.5,-1) {};
				\draw (x) to (y);
				\draw (x) to node[above]{$\alpha$} (x1);
				\draw (x) to node[below]{$\beta$} (x2);
				\draw (y) to node[above]{$\alpha$} (y1);
				\draw (y) to (y2);
				\draw (y) to (y3);
				\draw (y) to (y4);
				\draw (y) to (y6);
				\draw (y) to (y5);
				\draw [dotted] (y8) to (y7);
				\draw [dashed, , bend left] (y1) to node[right]{$N''(y)$} (y2);
				\draw [dashed, , bend left] (y4) to node[right]{$N'(y)$} (y6);
				\draw [dashed, , bend left] (y9) to node[right]{$S$} (y10);
			\end{tikzpicture}
			\caption{Neighborhood of the edge $xy$ in $G$ in Case~\ref{case:OneColorInFxy}}
			\label{fig:DegX3}
		\end{figure}
		
		\begin{subsubcase}
			Both the colors in $\{g(xx_1),g(xx_2)\}$ belong to $F_{xy}$.
		\end{subsubcase}
		
		Recall that we have $g(xx_1)=\alpha$ and $g(xx_2)=\beta$. Let $y_1$ and $y_2$ be the neighbors of $y$ such that $g(yy_1) = \alpha$ and $g(yy_2) = \beta$. Since the colors $\alpha$ and $\beta$ are seen at both the vertices $x$ and $y$ in $G'$, we have $|F_{xy} \cup F_{yx}| \le \Delta - 1$, implying that $|R| \geq 6$  for this case. Let $\{\mu_1,\mu_2,\mu_3,\mu_4,\mu_5,\mu_6\} \in R$. If any color in $R$ is valid for the edge $xy$, then we are done. Hence, we can assume that every candidate color for the edge $xy$ in $R$ is not valid. This implies that for every color $\mu_i$ in $R$, there exists either an $(\alpha,\mu_i,xy)$-critical path or a $(\beta,\mu_i,xy)$-critical path in $G'$ with respect to $g$ or both. Further, if there exists a color $\zeta \in T$ such that there is no $(\alpha,\zeta,xy)$-critical path and there is no $(\beta,\zeta,xy)$-critical path in $G'$ with respect to $g$, then we can free the color $\zeta$ if necessary and assign $\zeta$ to the edge $xy$ and thereby extend $g$ to the required coloring $f$ of $G$. Hence, we can also assume that for every color $\zeta \in T$, there exists either an $(\alpha,\zeta,xy)$-critical path or a $(\beta,\zeta,xy)$-critical path in $G'$ with respect to $g$ or both.
		
		Since we already have that at least one among $\alpha$ or $\beta$ belongs to $F_{xy} \setminus S$,  we are sure that at most one color in the set $\{\alpha,\beta\}$ is in $S$. This also implies that at least one color in the set $\{\alpha,\beta\}$ is not in $S$. Without loss of generality, assume that $\beta \notin S$. Then since $\beta \in F_{xy}$, we can infer that $\beta \in F_{xy} \setminus S$.
		
		Let us assume that $\beta \in F_{xx_1}$. Note that $|F_{xx_1}| \le \Delta - 1$. Since $\beta \in F_x$, we have $\beta \notin T$. This together with Claim~\ref{clm:AtLeastDeltaMinusOneCandidateColors} and the assumption that $\beta \in F_{xx_1}$ imply that there exists a color $\eta$ such that $\eta \in T$ but $\eta \notin F_{xx_1}$. Therefore, there can not be an $(\alpha,\eta,xy)$-critical path in $G'$ with respect to $g$ implying that there exists a $(\beta,\eta,xy)$-critical path, since $\eta \in T$. Hence, we can free the color $\eta$ and recolor the edge $xx_1$ with $\eta$ without forming any new bichromatic cycles, since the $(\beta,\eta)$-bichromatic path starting from the vertex $x$ can not reach $x_1$ because it ends at $y$. Now, since $\eta \notin F_{xy}$ and $\beta \in F_{xy}$, by Case~\ref{case:OneColorInFxy}, we are done.
		
		Now, assume that $\beta \notin F_{xx_1}$. Now, we have $|F_{xx_1} \cup \{\beta\}| \le \Delta$ and $|F_{xy} \setminus S| \le 3$. But since $\beta \in F_{xy} \setminus S$, we have $|F_{xx_1} \cup \{\beta\} \cup (F_{xy} \setminus S)| \le \Delta + 2$. Further, since $|S'| \le 2$, we have $|F_{xx_1} \cup \{\beta\} \cup (F_{xy} \setminus S) \cup S'| \le \Delta + 4$. Therefore, since we have a total of $\Delta + 5$ colors, we are sure that there exists a color $\gamma$ such that $\gamma \notin F_{xx_1} \cup \{\beta\} \cup (F_{xy} \setminus S) \cup S'$. Now, we free the color $\gamma$ and recolor the edge $xx_1$ with $\gamma$. This recoloring is valid since $\beta \notin F_{xx_1}$. Since $\gamma \notin F_{xy}$, and $\beta \in F_{xy}$, by Case~\ref{case:OneColorInFxy}, we are done.
		
		Therefore, in any case, we can extend the coloring $g$ of $G'$ to a coloring $f$ of $G$ with the same number of colors, which in turn confirms the validity of Theorem~\ref{thm:ACI3deg}.
	\end{proof}
	
	\section{Conclusion}
	We conclude our discussion on the acyclic chromatic index of degenerate graphs by reiterating Theorem~\ref{thm:ACIkdeg} and Theorem~\ref{thm:ACI3deg}. For any $k$-degenerate graph $G$, we have $a'(G) \le \lceil(\frac{k+1}{2})\Delta\rceil + 1$. Further, for any $3$-degenerate graph $G$, we have $a'(G) \le \Delta+5$. But the acyclic edge coloring conjecture gives an upper bound of $\Delta+2$ for any graph. Hence, one can take up the study of $3$-degenerate graphs and try to prove the conjecture for a $3$-degenerate graph. The same thing holds for a $k$-degenerate graph and one can try to improve the existing upper bound for the acyclic chromatic index of a $k$-degenerate graph which constitutes a nice research problem.
	
	\bibliographystyle{apa}
	\bibliography{Degenerate}
	
\end{document}